\documentclass[12pt]{amsart}

\usepackage{amsmath,amsfonts,amssymb}
\numberwithin{equation}{section}

\newtheorem{theorem}{Theorem}[section]
\newtheorem{lem}[theorem]{Lemma}

\newtheorem{proposition}[theorem]{Proposition}

 \DeclareMathOperator{\Tr}{Tr}
\DeclareMathOperator{\Ad}{Ad} \DeclareMathOperator{\ad}{ad}

\title [A contraction of the principal series...]{A contraction of the principal series
by Berezin-Weyl quantization}
\author {Benjamin Cahen}
\address{Universit\'e de Metz, UFR-MIM\\
D\'epartement de math\'ematiques\\
LMMAS, ISGMP-B\^at. A\\
Ile du Saulcy 57045, Metz cedex 01, France.}
\email{cahen@univ-metz.fr}


\begin{document}

\begin{abstract} We study a contraction of the principal series representations
of a noncompact semisimple Lie group to the unitary irreducible
representations of its Cartan motion group by means of the
Berezin-Weyl quantization on the coadjoint orbits associated with
these representations.

\end{abstract}

\maketitle
\section {Introduction}
In the pioneering paper \cite{IW}, In\"on\"u and Wigner introduced
the notion of contraction of Lie groups and of Lie group
representations on physical grounds: If two physical theories are
related by a limiting process, then the associated invariance groups
and their representations should be also related by a limiting
process called contraction. For example, the Galilei group is a
contraction, that is, a limiting case, of the Poincar\'e group and
the unitary irreducible representations of the Galilei group are
limits of unitary irreducible representations of the Poincar\'e
group \cite{IW}.

\par The systematic study of the contractions of Lie group representations
began with the work of Mickelsson and Niederle \cite{MN}. In
\cite{MN}, a proper definition of the contraction of unitary
representations of Lie groups was given for the first time. The
non-zero mass representations of the Euclidean group ${\mathbb
R}^{n+1}\rtimes SO(n+1)$ and the positive mass-squared
representations of the Poincar\'e group ${\mathbb R}^{n+1}\rtimes
SO_0(n,1)$ were obtained by contraction (i.e. as limits in the sense
defined in \cite{MN}) of the principal series representations of
$SO_0(n+1,1)$. These results were partially generalized by Dooley
and Rice in \cite{DoR1} and \cite{DoR2} by following an idea of
Mackey \cite{Ma}. In \cite{DoR2}, a contraction of the principal
series of a noncompact semisimple Lie group to the unitary
irreducible representations of its Cartan motion group was
established.

\par In fact, a contraction of Lie group representations provides
a link between the Harmonic Analysis on two different Lie groups. In
particular, contractions allow to recover some classical formulas of
the theory of special functions \cite{DoR1}, \cite{Ri}. Contractions
also permit to transfer results on $L^p$-multipliers from unitary
groups to Heisenberg groups \cite{DoG}, \cite{RiRu}.

\par In \cite{Do}, Dooley suggested interpreting contractions of
representations in the context of the Kirillov-Kostant method of
orbits \cite{Kir} and, in \cite{CDo}, Cotton and Dooley showed how
to describe contractions of representations by means of adapted Weyl
correspondences. The notion of adapted Weyl correspondence was
introduced in \cite{CaLMP} and \cite{CaCR}. Given a Lie group $G$
and a unitary irreducible representation $\pi$ of $G$ on a Hilbert
space $\mathcal H$, an adapted Weyl correspondence on a coadjoint
orbit $\mathcal O$ associated with $\pi$ by the Kirillov-Kostant
method of orbits is a linear isomorphism $W$ from a class of
functions on $\mathcal O$ (called symbols) onto a class of operators
on $\mathcal H$, which is adapted to $\pi$ in the following sense:
for each element $X$ of the Lie algebra of $G$, the function $\tilde
{X}$ defined on $\mathcal O$ by $\tilde {X}(\xi)=\langle
\xi,X\rangle $ is a symbol and the equality $W(i\tilde {X})=d\pi
(X)$ holds on a dense subspace of $\mathcal H$. A precise definition
of the notion of adapted Weyl correspondence can be found in
\cite{CaDGA}. Adapted Weyl correspondences have been constructed in
various situations, see the introduction of \cite{CaDGA}.

\par The
approach of \cite{CDo} is particularly efficient in the case when
the coadjoint orbits associated with the representations have
K\"ahlerian structures. In that case, the representation spaces are
reproducing kernel Hilbert spaces and the Berezin calculus generally
provides an adapted Weyl correspondence on the corresponding
coadjoint orbits \cite{CGR}. For example, in \cite{CaJAM},
\cite{CaCont} and \cite{CaCDS}, we used Berezin quantization in
order to establish contractions of the unitary irreducible
representations of a compact semisimple Lie group and of the
discrete series of a noncompact semisimple Lie group to the unitary
irreducible representations of a Heisenberg group.

\par In \cite{CDo}, the case of the contraction of the principal
series of $SL(2,{\mathbb R})$ to the unitary representations of
${\mathbb R}^2\rtimes SO(2)$ was treated by using the Weyl calculus.
In \cite{CaLie1}, the more complicated example of the contraction of
the principal series of $SO_0(n+1,1)$ to some unitary irreducible
representations of ${\mathbb R}^{n+1}\rtimes SO_0(n,1)$ was studied
similarly.

\par More generally, in the present paper, we apply the ideas of
\cite{CDo} to the study of the contraction of the principal series
representations of a noncompact semisimple Lie group $G$ to the
unitary irreducible representations of its Cartan motion group
$V\rtimes K$. We obtain very simple parametrizations of the
corresponding coadjoint orbits of $G$ and of $V\rtimes K$ by using
the method of \cite{CaDGA} which is based on the dequantization of
the representations by means of the Berezin-Weyl calculus introduced
in \cite{CaLMP}. This allows us to construct adapted Weyl
correspondences on these coadjoint orbits. Then we show how the
parametrizations of the orbits as well as the adapted Weyl
correspondences are related by the contraction process. In
particular, we get an infinitesimal version of the results of
\cite{DoR2} on the contraction of the principal series.

\par This paper is organized as follows. In Section 2 and Section 3,
we realize the representations of the principal series of $G$ and
the unitary irreducible representations of $V\rtimes K$ in
compatible ways and we compute the corresponding derived
representations. In Section 4, we introduce the Berezin-Weyl
calculus. In Section 5 and Section 6, we dequantize the
representations and then we obtain the parametrizations of the
associated orbits and the adapted Weyl correspondences. In Section
7, we recover a contraction result of \cite{DoR2} in the \lq
noncompact picture\rq \,(in the terminology of \cite{Kn}, Chapter
7). Finally, in Section 8, we show that the adapted Weyl
correspondences on the coadjoint orbits of $G$ and of $V\rtimes K$
associated with the representations are related by the contraction
process and we give a contraction result for the derived
representations.

\section {Principal series representations}
In this section, we first introduce some notation. Our main
references are \cite{Kn}, Chapter 7 and \cite{Wa}, Chapter 8. We
obtain a realization of the principal series representations which
is convenient for the study of contractions by modifying slightly
the standard \lq noncompact\rq \, realization \cite{Kn}, p. 169,
\cite{Wa}, Section 8.4.8 and we compute the corresponding derived
representations.

\par Let $G$ be a connected
noncompact semisimple real Lie group with finite center. Let
$\mathfrak g$ be the Lie algebra of $G$. We identify
$G$-equivariantly $\mathfrak g$ to its dual space ${\mathfrak
g}^{\ast}$ by using the Killing form $\beta$ of $\mathfrak g$
defined by $\beta (X,Y)=\Tr (\ad X \ad Y)$ for $X$ and $Y$ in
$\mathfrak g$. Let $\theta$ be a Cartan involution of $\mathfrak g$
and let $\mathfrak {g}=\mathfrak {k} \oplus V$ be the corresponding
Cartan decomposition of $\mathfrak g$. Let $K$ be the connected
compact (maximal) subgroup of $G$ with Lie algebra $\mathfrak k$.
Let $\mathfrak a$ be a maximal abelian subalgebra of $V$ and let $M$
be the centralizer of $\mathfrak a$ in $K$. Let $\mathfrak m$ denote
the Lie algebra of $M$. We can decompose $\mathfrak g$ under the
adjoint action of $\mathfrak a$:
\begin{equation*}\mathfrak{g} =\mathfrak{a} \oplus \mathfrak{m}
\oplus \sum_{\lambda\in \Delta}\mathfrak{g}_{\lambda}
\end{equation*}
where  $\Delta$ is the set of restricted roots. We fix a Weyl
chamber in $\mathfrak {a}$ and we denote by ${\Delta^+}$ the
corresponding set of positive roots. We set $\mathfrak{n}
=\sum_{\lambda \in {\Delta^+}}\mathfrak{g}_{\lambda}$ and $\bar
{\mathfrak n}=\sum_{\lambda \in {\Delta^+}}\mathfrak{g}_{-\lambda}$.
Then $\bar {\mathfrak n}=\theta ({\mathfrak n})$. Let $A$, $N$ and
$\bar N$ denote the analytic subgroups of $G$ with algebras
$\mathfrak {a}$, $\mathfrak {n}$ and $\bar {\mathfrak n}$,
respectively. We fix a regular element $\xi_1$ in $\mathfrak a$,
that is, $\lambda(\xi_1)\not = 0$ for each $\lambda\in \Delta$ and
an element $\xi_2$ in $\mathfrak m$. Let $\xi_0=\xi_1+\xi_2$. Denote
by $O(\xi_0)$ the orbit of $\xi_0$ in ${\mathfrak g}^{\ast}\simeq
{\mathfrak g}$ under the (co)adjoint action of $G$ and by $o(\xi_2)$
the orbit of $\xi_2$ in ${\mathfrak m}$ under the adjoint action of
$M$.

\par Let $\sigma$ be a unitary irreducible representation of $M$ on
a complex (finite-dimensional) vector space $E$. Henceforth we
assume that $\sigma$ is associated with the orbit $o(\xi_2)$ in the
following sense, see \cite{Wild}, Section 4. For a maximal torus $T$
of $M$ with Lie algebra $\mathfrak {t}$, $i\beta (\xi_2,\cdot)\in
i{\mathfrak t}^{\ast}$ is a highest weight for $\sigma$.

\par Now we consider the unitarily induced representation
\begin{equation*}\hat {\pi}=\rm {Ind}_{MAN}^{G}\,\left(\sigma \otimes \exp
({i\nu})\otimes 1_N\right)\end{equation*} where $\nu=\beta
(\xi_1,\cdot)\in {\mathfrak a}^{\ast}$. The representation $\hat
{\pi}$ lies in the unitary principal series of $G$ and is usually
realized on the space $L^2(\bar {N},E)$ which is the Hilbert space
completion of the space $C_0(\bar {N},E)$ of compactly supported
smooth functions $\phi:\bar{N}\rightarrow  E$ with respect to the
norm defined by
\begin{equation*}\Vert\phi\Vert^2=\int_{\bar N} \langle\phi\,(y)\,,\,\phi
(y)\rangle_E\,dy \end{equation*} where $dy$ is the Haar measure on
$\bar N$ normalized as follows. Let $(E_i)_{1\leq i\leq n}$ be an
orthonormal basis for $\bar {\mathfrak n}$ with respect to the
scalar product defined by $(\,Y,Z\,):=-\beta(Y,\theta (Z))$. Denote
by $(Y_1,Y_2,\ldots, Y_n)$ the coordinates of $Y\in \bar {\mathfrak
n}$ in this basis and let $dY=dY_1dY_2\ldots dY_n$ be the Euclidean
measure on $\bar {\mathfrak n}$. The exponential map $\exp$ is a
diffeomorphism from $\bar {\mathfrak n}$ onto $\bar {N}$ and we set
$dy=\log^{\ast}(dY)$ where $\log ={\exp }^{-1}$.

\par Recall that $\bar {N}MAN$ is a open dense subset of $G$. We
denote by $g=\bar {n}(g)m(g)a(g)n(g)$ the decomposition of $g\in
\bar {N}MAN$. For $g\in G$ the action of the operator ${\tilde
\pi}(g)$ is given by \begin{equation}\bigl({\tilde \pi
}(g)\phi\bigr)(y)=e^{-(\rho+i\nu)\log a(g^{-1}y)}{\sigma}\bigl
(m(g^{-1}y)\bigr)^{-1}\,\phi\bigr(
\bar{n}(g^{-1}y)\bigr)\end{equation} where $\rho (H):=\frac
{1}{2}\Tr _{\bar {\mathfrak n}} (\ad H )=\frac {1}{2}\sum_{\lambda
\in {\Delta^+}} \lambda$.

\par Recall that we have the Iwasawa decomposition $G=KAN$. We
denote by $g={\tilde k}(g){\tilde a}(g){\tilde n}(g)$ the
decomposition of $g\in G$.

\par In order to simplify the study of the contraction, we slightly
modify the preceding realization of $\hat \pi$ as follows. Let $I$
be the unitary isomorphism of $L^2(\bar {N},E)$ defined by
\begin{equation*} (I\phi)(y)=e^{-i\nu (\log {\tilde
a}(y))}\phi(y).\end{equation*} Then we introduce the realization
$\pi$ of $\hat {\pi}$ defined by $\pi (g):=I^{-1}\tilde {\pi}(g)I$
for each $g\in G$. We immediately obtain
\begin{align}\bigl({ \pi
}(g)\phi\bigr)(y)=e^{i\nu(\log {\tilde a}(y)-\log {\tilde a}({\bar
n}(g^{-1}y))}&e^{-(\rho+i\nu)\log a(g^{-1}y)}{\sigma}\bigl
(m(g^{-1}y)\bigr)^{-1}\\ &\phi\bigr( \bar{n}(g^{-1}y)\bigr).
\nonumber
\end{align}
Formula (2.2) can be simplified as follows. For $g\in G$ and $y\in
{\bar N}$, we can write
\begin{align*}&g^{-1}y={\bar n}(g^{-1}y)m(g^{-1}y)a(g^{-1}y)n(g^{-1}y)\\
&={\tilde k}({\bar n}(g^{-1}y)){\tilde a}({\bar n}(g^{-1}y)){\tilde
n}({\bar n}(g^{-1}y))m(g^{-1}y)a(g^{-1}y)n(g^{-1}y).\\ \end{align*}
Then we have
\begin{equation*}{\tilde a}(g^{-1}y)={\tilde a}({\bar
n}(g^{-1}y))a(g^{-1}y).\end{equation*} Hence we obtain
\begin{align}\bigl({ \pi
}(g)\phi\bigr)(y)=e^{i\nu(\log {\tilde a}(y)-\log {\tilde
a}(g^{-1}y))}&e^{-\rho(\log a(g^{-1}y))}{\sigma}\bigl
(m(g^{-1}y)\bigr)^{-1}\\ &\phi\bigr( \bar{n}(g^{-1}y)\bigr).
\nonumber
\end{align}
\par Now we give an explicit formula for the differential $d\pi$ of
$\pi$. Let us introduce some additional notation. If $H$ is a Lie
group and $X$ is an element of the Lie algebra of $H$ then we denote
by $X^+$ the right-invariant vector field generated by $X$, that is,
$X^+(h)=\frac {d}{dt}\,(\exp (tX))h \vert _{t=0}$ for $h\in H$. We
denote by $p_{\mathfrak a}$, $p_{\mathfrak m}$ and $p_{\bar
{\mathfrak n}}$ the projection operators of $\mathfrak g$ on
$\mathfrak a$, $\mathfrak m$ and $\bar {\mathfrak n}$ associated
with the decomposition ${\mathfrak g}={\bar {\mathfrak
n}}\oplus{\mathfrak m}\oplus{\mathfrak a}\oplus {\mathfrak n}$.
Moreover, we also denote by ${\tilde p}_{\mathfrak a}$ the
projection operator of $\mathfrak g$ on $\mathfrak a$ associated
with the decomposition ${\mathfrak g}= {\mathfrak k}\oplus{\mathfrak
a}\oplus {\mathfrak n}$. We need the following lemma.

\begin{lem} 1) For each $X \in {\mathfrak g}$ and each $y\in {\bar N}$, we have

\begin{align*}{\frac {d}{dt}}\,a(\exp(tX)y)\vert_{t=0}&=p_{\mathfrak{a}}(\Ad
(y^{-1})X)\\
{\frac {d}{dt}}\,m(\exp(tX)y)\vert_{t=0}&=p_{\mathfrak{m}}(\Ad
(y^{-1})X)\\
{\frac {d}{dt}}\,\bar {n}(\exp(tX)y)\vert_{t=0}&=\bigl(\Ad
(y)\,p_{\bar {\mathfrak{n}}}(\Ad (y^{-1})\,X)\bigr)^+(y).\\
\end{align*}

\noindent 2) For each $X\in {\mathfrak g}$ and each $g\in G$, we
have
\begin{equation*} {\frac {d}{dt}}\,{\tilde a}(\exp(tX)g)\vert_{t=0}=
\left({\tilde p}_{\mathfrak
a}\bigl(\Ad (\tilde {k}(g)^{-1})X\bigr )\right)^+({\tilde
a}(g)).\end{equation*}
\end{lem}
\begin{proof} To prove 1), we consider the diffeomorphism
$\mu:\,\bar{N}\times M\times A\times N \rightarrow \bar{N}MAN$
defined by $\mu (y,m,a,n)=yman$. For $y\in \bar {N}$, $Y\in \bar
{\mathfrak{n}}$, $U\in \mathfrak{m}$, $H\in \mathfrak{a}$ and $Z\in
\mathfrak{n}$, we have \begin{align} d\mu(y,e,e,e)&(Y^+(y),U,H,Z)
\\&=
\frac{d}{dt}\, \exp (tY)y\exp (tU)\exp (tH)\exp (tZ)\Bigl\vert_{t=0}\nonumber\\
&=\bigl(Y+ \Ad (y) (U+H+Z)\bigr)^+(y).\nonumber\\ \nonumber
\end{align} Now, let $X\in {\mathfrak g}$. We can write $\Ad
(y^{-1})X=Y_0+U+H+Z$ where $Y_0\in \bar {\mathfrak{n}}$, $U\in
\mathfrak{m}$, $H\in \mathfrak{a}$ and $Z\in \mathfrak{n}$. Then
Equality (2.4) implies that $d\bar{n}(y)(X^+(y))=(\Ad (y)
\,Y_0)^+(y)$. This proves the last equality of 1). The other
equalities are proved similarly. Finally, we prove 2) analogously.
\end{proof}
From this lemma, we immediately deduce the following proposition.

\begin{proposition} For $X\in {\mathfrak g}$,
$\phi \in C_0(\bar {N},E)$ and $y\in {\bar N}$, we have
\begin{align*}(d\pi&(X)\phi)(y)=i\nu \Bigl({\tilde p}_{\mathfrak
a}\bigl(\Ad (\tilde {k}(y)^{-1})X\bigr ) \Bigr)\phi(y)\\&+ \rho
\bigl( p_{\mathfrak{a}}(\Ad (y^{-1})X)\bigr)\phi (y)+d\sigma \bigl(
p_{\mathfrak{m}}(\Ad (y^{-1})X)\bigr)\phi (y)\\&-d\phi (y) \bigl(\Ad
(y)\,p_{\bar {\mathfrak{n}}}(\Ad (y^{-1})\,X)\bigr)^+(y).\\
\end{align*} \end{proposition}

\section {Representations of the Cartan motion group}
We retain the notation from Section 2. In particular, we have the
Cartan decomposition ${\mathfrak g}={\mathfrak k}\oplus V$ where $V$
is the orthogonal complement of ${\mathfrak k}$ in ${\mathfrak g}$
with respect to the Killing form $\beta$. We denote by
$p^c_{\mathfrak k}$ and $p^c_V$ the projections of ${\mathfrak g}$
on ${\mathfrak k}$ and $V$ associated with the Cartan decomposition.

\par We form the semidirect product $G_0:=V\rtimes K$. The group
law of $G_0$ is given by
\begin{equation*}(v,k).(v',k')=(v+\Ad (k)v',kk')\end{equation*}
for $v,\,v'$ in $V$ and $k,\,k' \in K$. The Lie algebra ${\mathfrak
g}_0$ of $G_0$ is the space $V\times {\mathfrak k}$ endowed with the
Lie bracket
\begin{equation*}[(w,U),(w',U')]_0=([U,w']-[U',w],[U,U'])
\end{equation*}
for $w,\,w'$ in $V$ and $U,\,U'$ in $\mathfrak k$.

\par Recall that $\beta$ is positive definite on $V$ and negative
definite on ${\mathfrak k}$ \cite{Hel}, p. 184. Then, by using
$\beta$, we can identify $V^{\ast}$ to $V$ and ${\mathfrak
k}^{\ast}$ to ${\mathfrak k}$, hence ${\mathfrak g}^{\ast}_0\simeq
V^{\ast}\times {\mathfrak k}^{\ast}$ to $V\times {\mathfrak k}$.
Under this identification, the coadjoint action of $G_0$ on
${\mathfrak g}^{\ast}_0\simeq V\times {\mathfrak k}$ is then given
by
\begin{equation*} (v,k)\cdot (w,U)=(\Ad(k)w, \Ad(k)U+[v,\Ad (k)w])
\end{equation*} for $v,\,w$ in $V$, $k$ in $K$ and $U$ in $\mathfrak
k$. This is a particular case of the general formula for the
coadjoint action of a semidirect product, see for instance
\cite{Rawn}.

\par We need the following lemma.
\begin{lem} For each regular element $\xi_1$ of ${\mathfrak a}$, the
space $\ad \xi_1\,(V)$ is the orthogonal complement of ${\mathfrak
m}$ in ${\mathfrak k}$. \end{lem}
\begin{proof} For each $\lambda \in {\Delta^+}$, let $E_{\lambda}\not= 0$
be in ${\mathfrak g}_{\lambda}$. Note that the space $p^c_{\mathfrak
k}({\mathfrak n})=p^c_{\mathfrak k}({\bar {\mathfrak n}})$ is
generated by the elements $E_{\lambda}+\theta (E_{\lambda})$ and
hence orthogonal to $\mathfrak m$. Now, by applying $p^c_{\mathfrak
k}$ to the decomposition ${\mathfrak g}={\mathfrak m}+{\mathfrak
a}+{\mathfrak n}+{\bar {\mathfrak n}}$ we get ${\mathfrak
k}={\mathfrak m}+p^c_{\mathfrak k}({\mathfrak n})$. This shows that
$p^c_{\mathfrak k}({\mathfrak n})$ is the orthogonal complement of
${\mathfrak m}$ in ${\mathfrak k}$. On the other hand, by applying
$p^c_V$ to the preceding decomposition of $\mathfrak g$, we obtain
$V={\mathfrak a}+p^c_V({\mathfrak n})$. Since $p^c_V({\mathfrak n})$
is generated by the elements $E_{\lambda}-\theta(E_{\lambda})$, the
space $\ad \xi_1\,(V)$ is then generated by the elements
\begin{equation*}\ad \xi_1\,(E_{\lambda}-\theta(E_{\lambda}))=\lambda
(\xi_1)(E_{\lambda}+\theta(E_{\lambda}))\end{equation*} where
$\lambda (\xi_1)\not=0$ for $\lambda \in {\Delta}$. Hence $\ad
\xi_1\,(V)=p^c_{\mathfrak k}({\mathfrak n})$ is the orthogonal
complement of $\mathfrak m$ in $\mathfrak k$.
\end{proof}
\par The coadjoint orbits of the semidirect product of a Lie group
by a vector space were described by Rawnsley in \cite{Rawn}. For
each $(w,U)\in {\mathfrak g}^{\ast}_0\simeq {\mathfrak g}_0$, we
denote by $O(w,U)$ the orbit of $(w,U)$ under the coadjoint action
of $G_0$. The following lemma shows that, for almost all $(w,U)$,
the orbit $O(w,U)$ is of the form $O(\xi_1,\xi_2)$ with $\xi_1\in
{\mathfrak a}$ and $\xi_2\in {\mathfrak m}$.

\begin{lem} 1) Let $\mathcal O$ be a coadjoint orbit for the
coadjoint action of $G_0$ on ${\mathfrak g}^{\ast}_0\simeq
{\mathfrak g}_0$. Then there exists an element of $\mathcal O$ of
the form $(\xi_1, U)$ with $\xi_1\in{\mathfrak a}$. Moreover, if
$\xi_1$ is regular then there exists $\xi_2\in {\mathfrak m}$ such
that $(\xi_1,\xi_2)\in {\mathcal O}$.

\noindent 2) Let $\xi_1$ be a regular element of $\mathfrak a$. Then
$M$ is the stabilizer of $\xi_1$ in $K$. \end{lem}
\begin{proof} 1) Let $(w,U)\in {\mathcal O}$. For each $k\in K$ we have
$(0,k)\cdot (w,U)=(\Ad(k)w,\Ad(k)U)$. By \cite{Kn}, p. 120, we can
choose $k\in K$ so that $\Ad(k)w\in {\mathfrak a}$. We set
$\xi_1:=\Ad(k)w$. If we assume that $\xi_1$ is regular then by Lemma
3.1 we can write $U=\xi_2+[\xi_1,v]$ where $\xi_2\in {\mathfrak m}$
and $v\in V$. Then $(\xi_1,U)=(v,e)\cdot (\xi_1,\xi_2)$. Hence
${\mathcal O}=O(\xi_1,\xi_2)$.

\noindent 2) Denote by $K(\xi_1)$ the stabilizer of $\xi_1$ in $K$
and by ${\mathfrak g}(\xi_1)$ the centralizer of $\xi_1$ in
$\mathfrak g$. Since $\xi_1$ regular, we have ${\mathfrak
g}(\xi_1)={\mathfrak a}\oplus {\mathfrak m}$ \cite{Hel}, p. 263.

Let $k\in K(\xi_1)$. Then $\Ad (k)$ leaves ${\mathfrak g}(\xi_1)$
invariant. Thus ${\mathfrak g}(\xi_1)\cap V={\mathfrak a}$ is also
invariant under $\Ad (k)$. Hence $k\in M$. This shows that
$K(\xi_1)\subset M$. Finally $K(\xi_1)=M$.
\end{proof}

\par In the rest of the section, we consider the orbit
$O(\xi_1,\xi_2)$ of $(\xi_1,\xi_2)\in {\mathfrak a}\times {\mathfrak
m}\subset {\mathfrak g}^{\ast}_0\simeq {\mathfrak g}_0$ under the
coadjoint action of $G_0$. As in Section 2, we assume that $\xi_1$
is a regular element of $\mathfrak a$ and that the adjoint orbit
$o(\xi_2)$ of $\xi_2$ in $\mathfrak m$ is associated with a unitary
irreducible representation $\sigma$ of $M$ which is realized on a
(finite-dimensional) Hilbert space $E$. Then $O(\xi_1,\xi_2)$ is
associated with the unitarily induced representation
\begin{equation*}\hat {\pi}_0=\rm {Ind}_{V\times
M}^{G_0}\,\left(e^{i\nu}\otimes \sigma \right)\end{equation*} where
$\nu=\beta(\xi_1,\cdot)\in {\mathfrak a}^{\ast}$ (see \cite{Ko2} and
\cite{Rawn}). By a result of Mackey, $\hat {\pi}_0$ is irreducible
since $\sigma$ is irreducible \cite{Tay}.
\par Let $O_V(\xi_1)$ be the orbit of $\xi_1$ in $V$ under the
action of $K$. We denote by $\mu$ the $K$-invariant measure on
$O_V(\xi_1)\simeq K/M$. We denote by $\tilde {\pi}_0$ the usual
realization of $\hat {\pi}_0$ on the space of square-integrable
sections of a Hermitian vector bundle over $O_V(\xi_1)$ \cite{Ko2},
\cite{Sim}, \cite{Rawn}. Let us  briefly describe the construction
of $\tilde {\pi}_0$. We introduce the Hilbert $G_0$-bundle
$L:=G_0\times_{e^{i\nu}\otimes \sigma}E$ over $O_V(\xi_1)\simeq
K/M$. Recall that an element of $L$ is an equivalence class
\begin{equation*}[g,u]=\{(g.(v,m),e^{-i\nu(v) }
\sigma (m)^{-1}u)\,:\,v\in V,\,m\in M\} \end{equation*} where $g\in
G_0$, $u\in E$ and that $G_0$ acts on $L$ by left translations:
$g\,[g',u]:=[g.g',u]$. The action of $G_0$ on $O_V(\xi_1)\simeq K/M$
being given by $(v,k).\xi=\Ad(k)\xi$, the projection map
$[(v,k),u]\rightarrow \Ad(k)\xi_1$ is $G_0$-equivariant. The
$G_0$-invariant Hermitian structure on $L$ is given by
\[\langle [g,u],[g,u']\rangle =\langle u,u'\rangle _E \]
where $g\in G_0$ and $u,\,u'\in E$. Let ${\mathcal H}_0$ be the
space of sections $s$ of $L$ which are square-integrable with
respect to the measure $\mu$, that is,
\begin{equation}\Vert s\Vert_{{\mathcal
H}_0}^2=\int_{O_V(\xi_1)}\,\langle s(\xi)\,,\,s(\xi)\rangle \,d\mu
(\xi)<+\infty. \nonumber
\end{equation}
Then $\tilde{\pi}_0$ is the action of $G_0$ on ${\mathcal H}_0$
defined by
\begin{equation}(\tilde{\pi}_0(g)\,s)(\xi)=g\,s(g^{-1}.\xi).\nonumber
\end{equation}
\par For the study of contractions, it is more convenient to realize
$\hat{\pi}_0$ in the Hilbert space $L^2({\bar N},E)$ introduced in
Section 2. To this aim, we consider the map $\tau: y\rightarrow
\Ad({\tilde k}(y))\xi_1$ which is a diffeomorphism from $\bar N$
onto a dense open subset of $O_V(\xi_1)$ \cite{Wa}, Lemma 7.6.8. We
denote by $k\cdot y$ the action of $k\in K$ on $y\in {\bar N}$
defined by $\tau (k\cdot y)=\Ad(k)\tau (y)$ or, equivalently, by
$k\cdot y={\bar n}(ky)$. Then the $K$-invariant measure on $\bar N$
is given by $(\tau^{-1})^{\ast}(\mu)=e^{-2\rho(\log {\tilde
a}(y))}dy$ \cite{Wa}, Lemma 7.6.8. We associate with each $s\in
{\mathcal H}_0$ the function $\phi_s: {\bar N}\rightarrow E$ defined
by
\begin{equation*}s(\tau (y))=[(0,{\tilde k}(y))\,,\,e^{\rho(\log {\tilde
a}(y))}\phi_s(y)].\end{equation*}For $s$ and $s'$ in ${\mathcal
H}_0$, we have
\begin{equation*}\langle s(\tau (y)),s'(\tau(
y))\rangle =e^{2\rho(\log {\tilde a}(y))}\,\langle
\phi_s(y)\,,\,\phi_{s'}(y)\rangle _E.
\end{equation*} This implies that
\begin{equation*}\langle s\,,\,s'\rangle _{{\mathcal
H}_0}=\int_{{\bar N}}\,\langle \phi_s(y)\,,\,\phi_{s'}(y)\rangle
_E\,\,dy.
\end{equation*}
Moreover, for $s \in {\mathcal H}_0$, $g=(v,k)\in G_0$ and $y\in
{\bar N}$, we have
\begin{align*}(&{\tilde \pi}_0(g)s)(\tau
(y))=g\,s(g^{-1}.y)=g\,s(\tau (k^{-1}\cdot y))\\&=(v,k)\,[(0,{\tilde
k}(k^{-1}\cdot y))\,,\,e^{\rho(\log {\tilde a}(k^{-1}\cdot y)}\phi_s
(k^{-1}\cdot y)]\\&=[(v,k{\tilde k}(k^{-1}\cdot y))\,,\,e^{\rho(\log
{\tilde a}(k^{-1}\cdot y)}\phi_s (k^{-1}\cdot y)]\\&=e^{\rho(\log
{\tilde a}(k^{-1}\cdot y)}\,[(0,{\tilde k}(y)).(\Ad ({\tilde
k}(y))^{-1}v,m(k,y))\,,\,\phi_s (k^{-1}\cdot y)]\\&=e^{\rho(\log
{\tilde a}(k^{-1}\cdot y))+i\nu (\Ad ({\tilde
k}(y))^{-1}v)}\,[(0,{\tilde k}(y))\,,\,\sigma (m(k,y)) \phi_s
(k^{-1}\cdot y)]
\end{align*}
where we have set $m(k,y):={\tilde k}(y)^{-1}k{\tilde k}(k^{-1}\cdot
y)\in M$.  Hence we see that the equality
\begin{equation} (\pi_0(v,k)\phi) (y)=e^{i\beta (\Ad ({\tilde
k}(y))\xi_1,v)+\rho (\log {\tilde a}(k^{-1}\cdot y)-\log {\tilde
a}(y))}\,\sigma (m(k,y))\phi (k^{-1}\cdot y) \end{equation} defines
a unitary representation $\pi_0$ of $G_0$ on $L^2({\bar N},E)$ which
is unitarily equivalent to ${\tilde \pi}_0$, the intertwining
operator between $\pi_0$ and ${\tilde \pi}_0$ being $s\rightarrow
\phi_s$.
\par We can simplify Formula (3.1) as follows. Let $k\in K$ and
$y\in {\bar N}$. Write $k^{-1}y={\bar
n}(k^{-1}y)m(k^{-1}y)a(k^{-1}y)n(k^{-1}y)$. Then $k^{-1}{\tilde
k}(y)={\tilde k}({\bar n}(k^{-1}y))m(k^{-1}y)$. Thus $m(k,y)={\tilde
k}(y)^{-1}k{\tilde k}(k^{-1}\cdot y)=m(k^{-1}y)^{-1}$. We also see
that
\begin{equation*}{\tilde a}(y)={\tilde a}(k^{-1}y)={\tilde a}
({\bar n}(k^{-1}y))a(k^{-1}y)={\tilde a}(k^{-1}\cdot y)
a(k^{-1}y).\end{equation*} Hence we obtain
\begin{equation} (\pi_0(v,k)\phi) (y)=e^{-\rho (\log
a(k^{-1}y)+i\beta (\Ad ({\tilde k}(y))\xi_1,v)} \,\sigma
(m(k^{-1}y))^{-1}\phi ({\bar n}(k^{-1}y)).
\end{equation}
The computation of $d\pi_0$ is quite similar to that of $d\pi$ (see
Section 2). By using Lemma 2.1 we easily obtain the following
result.

\begin{proposition} For $(v,U)\in {\mathfrak g}_0$, $\phi \in
C_0({\bar N},E)$ and $y\in {\bar N}$, we have
\begin{align*}(d\pi_0&(v,U)\phi)(y)=i\beta \bigl(\Ad ({\tilde k}(y))
\xi_1,v \bigr)\phi(y)\\
&+ \rho \bigl( p_{\mathfrak{a}}(\Ad (y^{-1})U)\bigr)\phi (y)+d\sigma
\bigl( p_{\mathfrak{m}}(\Ad (y^{-1})U)\bigr)\phi (y)\\&-d\phi (y)
\bigl(\Ad
(y)\,p_{\bar {\mathfrak{n}}}(\Ad (y^{-1})\,U)\bigr)^+(y).\\
\end{align*} \end{proposition}

\section {Berezin-Weyl calculus}
In this section, we keep the notation of the previous sections. We
recall some properties of the Berezin calculus on $o(\xi_2)$ and of
the Berezin-Weyl calculus on ${\bar N}\times {\bar {\frak n}}\times
o(\xi_2)$ which was introduced in \cite{CaLMP} as a generalization
of the usual Weyl calculus.
\par The Berezin calculus on $o(\xi_2)$ associates
with each operator $B$ on the finite-dimensional complex vector
space $E$ a complex-valued function $s(B)$ on the orbit $o(\xi_2)$,
which is called the symbol of $B$ (see \cite{Be}). The following
properties of the Berezin calculus can be found in \cite{CGR},
\cite{CaLMP}, \cite{CaComm}. \begin{proposition}
\begin{enumerate}
\item The map $B\rightarrow s(B)$ is injective.

\item  For each operator $B$ on $E$, we have $s(B^{\ast})=\overline
{s(B)}$.

\item  For each
operator $B$ on $E$, each $m\in M$ and each $\varphi \in o(\xi_2)$,
we have
\begin{equation*} s(B)(\Ad (m)\varphi)=s(\sigma (m)^{-1}B \sigma
(m))(\varphi).\end{equation*}

\item  For $X\in {\mathfrak m}$ and $\varphi \in o(\xi_2)$,
we have $s(d\sigma(X))(\varphi)=i\beta ( \varphi, X) $.
\end{enumerate}

\end{proposition}

In particular, note that the map $s^{-1}$ is an adapted Weyl
transform on $o(\xi_2)$ in the sense of \cite{CaDGA} (see also
Section 5).

\par Now we introduce the Berezin-Weyl calculus on ${\bar N}\times {\bar {\mathfrak
n}} \times o(\xi_2)$ as a slight modification of the usual Weyl
calculus for ${\rm End}(E)$-valued functions \cite{Ho}. We say that
a smooth function $f:\,(y,Z,\varphi)\rightarrow f(y,Z,\varphi)$ is a
symbol on ${\bar N}\times {\bar {\mathfrak n}} \times o(\xi_2)$ if
for each $(y,Z)\in {\bar N}\times {\bar {\mathfrak n}}$ the function
$\varphi \rightarrow f(y,Z,\varphi)$ is the symbol, in the Berezin
calculus on $o(\xi _2)$, of an operator on $E$ denoted by ${\hat
f}(y,Z)$. A symbol $f$ on ${\bar N}\times {\bar {\mathfrak n}}
\times o(\xi_2)$ is called an S-symbol if the function $\hat {f}$
belongs to the Schwartz space of rapidly decreasing smooth functions
on ${\bar N}\times {\bar {\mathfrak n}}$ with values in ${\rm End}
(E)$. For each S-symbol $f$ on ${\bar N}\times {\bar {\mathfrak n}}
\times o(\xi_2)$ we define an operator $ W(f)$ on $L^2({\bar N},E)$
by
\begin{equation}( W(f)\phi )(y)={(2\pi)}^{-n}\,\int_{{\bar {\mathfrak n}}
\times {\bar {\mathfrak n}}}\, e^{i\langle T,Z\rangle }{\hat
{f}}\bigl( y\exp(T/2), Z \bigr)\,\phi (y\exp T)\,dT\,dZ
\end{equation} for $\phi \in C_0({\bar N},E)$.

\par As the usual Weyl calculus, the Weyl-Berezin calculus can be
extended to much larger classes of symbols. Here we only consider a
class of polynomial symbols. For $Z\in {\bar {\mathfrak n}}$, we
denote by $(z_1,z_2,\ldots,z_n)$ the coordinates of $Z$ in the basis
$(E_i)_{1\leq i\leq n}$ of ${\bar {\mathfrak n}}$. We say that a
symbol $f$ on ${\bar N}\times {\bar {\mathfrak n}} \times o(\xi_2)$
is a P-symbol if the function ${\hat {f}}(y,Z)$ is polynomial in
$z_1,z_2,\ldots,z_n$. Let $f$ be the P-symbol defined by $
f(y,Z,\varphi)=u(y)z_1^{\alpha_1}z_2^{\alpha_2}\ldots
z_n^{\alpha_n}$ where $u\in C^{\infty}({\bar N})$. By imitating
\cite{Vo}, p. 105, we get
\begin{equation}({W}(f)\phi )(y)=(i{\partial_{z_1}})^{\alpha_1}
(i{\partial_{z_2}})^{\alpha_2}\ldots (i{\partial_{z_n}})^{\alpha_n}
\bigl( u(y\exp Z/2)\,\phi (y\exp Z) \bigr) \Bigl \vert
_{Z=0}.\end{equation} In particular, if $f(y,Z,\varphi)=u(y)$ where
$u\in C^{\infty}({\bar N})$, then
\begin{equation}( W(f)\phi )(y)=u(y)\,\phi
(y)\end{equation} and if $ f(y,Z,\varphi)=(v(y), Z)$ where $v\in
C^{\infty}({\bar N},{\bar {\mathfrak n}})$, then
\begin{align}(W(f)&\phi )(y)=i\Bigl( \sum_{k=1}^n {\frac {d}{dt}}
\bigl(E_k\,,\, v(y\exp(tE_k/2)\bigr)\bigr \vert_{t=0}\phi (T)
\\&+{\frac {d}{dt}}\,\phi (y\exp (tv(y))\bigr \vert_{t=0} \Bigr).\nonumber
\end{align}
The following lemmas will be needed in Section 5 and Section 6.

\begin{lem} Let $X\in {\mathfrak g}$ and let $f$ be the P-symbol on
${\bar N}\times {\bar {\mathfrak n}} \times o(\xi_2)$ defined by
$f(y,Z,\varphi)=(p_{\bar {\mathfrak n}}(\Ad (y^{-1})X),Z)$. Then we
have \begin{align*} W(f)\phi &(y)=-i\rho (p_{\mathfrak a}(\Ad
(y^{-1})X))\,\phi (y)\\&+i(d\phi) (y)\bigl(\Ad (y) p_{\bar
{\mathfrak n}}(\Ad (y^{-1})X)\bigr)^+(y)\end{align*} for each $\phi
\in C_0({\bar N}, E)$.
\end{lem}
\begin{proof} We apply Formula (4.4) to
$f(y,Z,\varphi)=(p_{\bar {\mathfrak n}}(\Ad (y^{-1})X),Z)$. On the
one hand, we have
\begin{align*} \sum_{k=1}^n {\frac {d}{dt}} \bigl(E_k\,,\,
v(y\exp(tE_k/2)\bigr)\bigr \vert_{t=0}&=-{\frac
{1}{2}}\,\sum_{k=1}^n \bigl(E_k,p_{\bar {\mathfrak n}}\bigl(\ad E_k
(\Ad (y^{-1})X)\bigr)\bigr)\\&={\frac {1}{2}}\,\Tr_{\bar {\mathfrak
n}}(p_{\bar {\mathfrak n}}\circ \ad (\Ad (y^{-1})X)).\end{align*}
But for each $Y\in {\mathfrak g}$, we have $\Tr_{\bar {\mathfrak
n}}(p_{\bar {\mathfrak n}}\circ \ad (Y))=-2\rho(p_{\mathfrak
a}(Y))$. This equality can be proved as follows. If $Y\in {\bar
{\mathfrak n}}$ then $p_{\bar {\mathfrak n}}\circ \ad Y =\ad Y$ is a
nilpotent endomorphism of ${\bar {\mathfrak n}}$. Thus $\Tr _{\bar
{\mathfrak n}}\left( p_{\bar {\mathfrak n}}\circ \ad Y \right)=0$.
If $Y\in {\mathfrak n}$ then, since $[{\mathfrak n}\,,\,{\mathfrak
g}_{\lambda}]\subset {\mathfrak a}+\sum_{\mu>\lambda}{\mathfrak
g}_{\mu}$ for each $\lambda<0$, we also have that $\Tr _{\bar
{\mathfrak n}}\left( p_{\bar {\mathfrak n}}\circ \ad Y \right)=0$.
If $X\in {\mathfrak m}$ then $p_{\bar {\mathfrak n}}\circ \ad Y=\ad
Y$ is an endomorphism of $\bar {\mathfrak n}$ which is
skew-symmetric with respect to $(\cdot,\cdot)$. Thus $\Tr _{\bar
{\mathfrak n}}\left( p_{\bar {\mathfrak n}}\circ \ad Y \right)=0$.
Finally, if $Y\in {\mathfrak a}$ then $\Tr _{\bar {\mathfrak
n}}\left( p_{\bar {\mathfrak n}}\circ \ad Y \right)=\Tr _{\bar
{\mathfrak n}}(\ad Y)=-2\rho (Y)$.
\par Then we get
\begin{equation}\sum_{k=1}^n {\frac {d}{dt}}
\bigl(E_k\,,\, v(y\exp(tE_k/2))\bigr)\bigr \vert_{t=0}=-\rho \bigl(
p_{\mathfrak a}(\Ad(y^{-1})X) \bigr).\end{equation} On the other
hand, we have
\begin{equation}{\frac {d}{dt}}\,\phi (y\exp (tv(y)))\bigr \vert_{t=0}
=(d\phi)(y)\bigl( \Ad(y)v(y)\bigr)^+(y).\end{equation} Putting (4.5)
and (4.6) together, we get the desired result.
\end{proof}

\par We can identify the cotangent bundle $T^{\ast}{\bar N}$ with ${\bar
N}\times {\bar {\mathfrak n}}$ by using the map $j:\,{\bar N}\times
{\bar {\mathfrak n}}\rightarrow T^{\ast}{\bar N}$ defined by
\begin{equation*} \langle j(y,Z),Y^+(y)\rangle=-\beta (\theta (Z),\Ad
(y^{-1})Y) \end{equation*} for $y\in {\bar N}$ and $Y,\,Z\in {\bar
{\mathfrak n}}$. Under this identification, the Liouville 1-form on
$T^{\ast}{\bar N}$ corresponds to the 1-form $\alpha$ on ${\bar
N}\times {\bar {\mathfrak n}}$ given by \begin{equation*}
\alpha_{(y,Z)}(Y^+(y),T)=-\beta (\theta (Z),\Ad (y^{-1})Y)
\end{equation*} for $y\in {\bar N}$ and $Z,\,Y,\,T \in
{\bar {\mathfrak n}}$. We denote by $\{\cdot,\cdot \}_1$ the Poisson
bracket associated with the symplectic 2-form $d\alpha$ on ${\bar
N}\times {\bar {\mathfrak n}}$. We also denote by $\omega_2$ the
Kirillov 2-form on $o(\xi_2)$ and by $\{\cdot,\cdot \}_2$ the
corresponding Poisson bracket. We form the symplectic product ${\bar
N}\times {\bar {\mathfrak n}}\times o(\xi_2)$ and denote by
$\{\cdot,\cdot \}_p$ the Poisson bracket associated with the
symplectic form $\omega_p:=d\alpha \otimes \omega_2$. Let $u,\,v\in
C^{\infty}({\bar N}\times {\bar {\mathfrak n} })$ and $a,\,b\in
C^{\infty}(o(\xi_2))$. Then, for $f(y,Z,\varphi)= u(y,Z)a(\varphi)$
and $g(y,Z,\varphi)=v(y,Z)b(\varphi)$ we have
\begin{equation*}\{f\,,\,g\}_p=u(y,Z)v(y,Z) \{a,\,b\}_2+a(\varphi ) b(\varphi )
\{u\,,\,v\}_1.\end{equation*}

\begin{lem}Let
$f$ and $g$ be two {\it P}-symbols on ${\bar N}\times
{\bar{\mathfrak n}}\times o(\xi_2)$  of the form
\begin{equation*}u(y)+\beta (v(y), \varphi ) +\sum_{k=1}^n\,w_k(y)z_k
\end{equation*}
 where
$u\in C^{\infty}({\bar N})$, $v\in C^{\infty}({\bar N}, {\bar
{\mathfrak n }})$ and $w_k\in C^{\infty}({\bar N})$ for $k=1, 2,
\ldots ,n$. Then we have \begin{equation*} [W(f)\,,\, W(g)]=-i\,
W(\{f\,,\,g\}_p).\end{equation*}\end{lem}

\begin{proof} We can prove this lemma by a case-by-case
verification. The computations are easy but tedious. For instance,
take $f(y,Z,\varphi)=w(y)z_k$ and $g(y,Z,\varphi)=w'(y)z_l$. For
$Y\in {\bar {\mathfrak n}}$ and $u\in C^{\infty}({\bar N})$, we set
$Y(u)(y):=\frac {d}{dt}u(y\exp (tY))\vert_{t=0}$ for each $y\in
{\bar N}$. We can easily verify that
\begin{equation*}\{f,g\}_p=-E_k(w')(y)w(y)z_l+E_l(w)(y)w'(y)z_k+w(y)w'(y)
\{z_k,z_l\}\end{equation*} where $\{z_k,z_l\}=\beta
(\theta(Z),[E_k,E_l])$. This implies that
\begin{align*} W(-i&\{f\,,\,g\})\phi=\frac {1}{2}E_lE_k(w)w'\phi-\frac
{1}{2}wE_kE_l(w')\phi\\&-E_k(w')wE_l(\phi)+E_l(w)w'E_k(\phi)-ww'[E_k,E_l](\phi)
\end{align*} which is precisely $[W(f),W(g)]$. The calculations in
the other cases are similar.
\end{proof}

\section {Adapted Weyl correspondence for $\pi$}
In this section, we first compute the Berezin-Weyl symbol of the
operator $-id\pi(X)$ for $X\in {\mathfrak g}$. This dequantization
process allows us to obtain an explicit symplectomorphism from
${\bar N}\times {\bar{\mathfrak n}}\times o(\xi_2)$ onto a dense
open subset of the orbit $O(\xi_0)$ and then to construct an adapted
Weyl correspondence on $O(\xi_0)$.

\begin{proposition} Let $\Psi$ be the map from ${\bar N}\times
{\bar{\mathfrak n}}\times o(\xi_2)$ to $\mathfrak g$ defined by
\begin{equation*}\Psi (y,Z,\varphi)=\Ad ({\tilde k}(y))\,\xi_1+\Ad
(y)(\varphi-\theta (Z)).\end{equation*} Then, for each $X\in
{\mathfrak g}$, the Berezin-Weyl symbol of the operator $-id\pi (X)$
is the P-symbol $f_X$ defined by
\begin{equation*}f_X(y,Z,\phi)=\beta (\Psi
(y,Z,\varphi)\,,\,X)).\end{equation*} \end{proposition}

\begin{proof} Let $X\in
{\mathfrak g}$. Recall that an explicit expression for $-id\pi (X)$
was given in Proposition 2.2. Then, by Proposition 4.1 and Lemma
4.2, we immediately see that the Berezin-Weyl symbol of the operator
$-id\pi (X)$ is the function $f_X$ defined by
\begin{align}f_X(y,Z,\phi)&=(p_{\bar {\mathfrak n}}\bigl(\Ad
(y^{-1})X),Z\bigr)+\beta \bigl(\varphi,p_{\mathfrak m}\bigl(\Ad
(y^{-1})X\bigr) \bigr)\\&+\nu \bigl( {\tilde p}_{\mathfrak a}(\Ad
({\tilde k}(y)^{-1})X \bigr).\nonumber \end{align} Now, let
$(y,Z,\phi)\in {\bar N}\times {\bar{\mathfrak n}}\times o(\xi_2)$.
Since the map $X\rightarrow f_X(y,Z,\phi)$ is linear, there exists
an element $\Psi (y,Z,\varphi)$ in $\mathfrak g$ such that
$f_X(y,Z,\phi)=\beta (\Psi (y,Z,\varphi)\,,\,X)$ for each $X\in
{\mathfrak g}$. More precisely, by using Equality (5.1) we get
\begin{align*} f_X&(y,Z,\phi)=-\beta \bigl(\Ad (y^{-1})X,\theta
(Z)\bigr) +\beta \bigl(\Ad (y^{-1})X,\varphi \bigr)\\&+\beta \bigl(
\xi_1,\Ad ({\tilde k}(y)^{-1})X\bigr)=\beta \bigl(\Ad ({\tilde
k}(y))\,\xi_1+\Ad (y)(\varphi-\theta (Z)),X\bigr).\end{align*} This
gives the desired result.
\end{proof}
\par We denote by $\omega$ the Kirillov 2-form on $O(\xi_0)$ and by
$\{\cdot,\cdot\}$ the corresponding Poisson bracket. Let  ${\tilde
O}(\xi_0)$ denote the dense open subset $\Ad ({\bar N}MAN)\xi_0$ of
$\mathfrak g$.

\begin{proposition} The map $\Psi$ is a symplectomorphism from
$({\bar N}\times {\bar{\mathfrak n}}\times
o(\xi_2),\omega_p=d\alpha\otimes \omega_2)$ onto $({\tilde
O}(\xi_0),\omega )$.\end{proposition}

\begin{proof} By \cite{CaLMP}, Proposition 1 and \cite{CaBeit1}, Proposition 4.3,
the map $\Psi_1$ from ${\bar N}\times {\bar{\mathfrak n}}\times
o(\xi_2)$ onto ${\tilde O}(\xi_0)$ defined by
\begin{equation*}\Psi_1 (y,Z,\varphi)=\Ad (y)(\xi_1+\varphi-\theta
(Z))\end{equation*} is a diffeomorphism. Note that if
\begin{equation*}Z'=Z+\theta \bigl( \Ad ({\tilde
n}(y)^{-1})\xi_1-\xi_1\bigr)\end{equation*} then
\begin{align*}\Psi (y,Z,\varphi)&=\Ad
({\tilde k}(y))\xi_1+\Ad (y) \bigl( \varphi-\theta (Z')-\Ad( {\tilde
n}(y)^{-1})\xi_1+\xi_1\bigr)\\&=\Ad ({\tilde k}(y))\xi_1-\Ad(
y\,{\tilde n}(y)^{-1})\xi_1+\Psi_1(y,Z,\varphi)\end{align*} Thus,
since $y\,{\tilde n}(y)^{-1}={\tilde k}(y){\tilde a}(y)$, we obtain
$\Psi (y,Z,\varphi)=\Psi_1(y,Z',\varphi)$. Hence $\Psi$ is a
diffeomorphism from ${\bar N}\times {\bar{\mathfrak n}}\times
o(\xi_2)$ onto ${\tilde O}(\xi_0)$.
\par Now, we show that $\Psi$ is also a symplectomorphism by
following the method of \cite{CaDGA}, Theorem 6.3. Recall that for
$X\in {\mathfrak g}$, $\tilde X$ denotes the function on $O(\xi_0)$
defined by ${\tilde X}(\xi)=\beta (\xi, X)$. Observe that $f_X\circ
\Psi={\tilde X}$. Let $X$ and $Y$ in $\mathfrak g$. One the one
hand, by Proposition 5.1 and Lemma 4.3, we have
\begin{equation*}[ W(f_X)\,,\,  W(f_Y)]=-i
W(\{f_X\,,\,f_Y\}_p).\end{equation*}  On the other hand, we have
\begin{equation*}[ W(f_X)\,,\, W(f_Y)]=[-id\pi (X),-id\pi
(Y)]=-d\pi ([X\,,\,Y])= -i  W(f_{[X,Y]}).\end{equation*} Then we get
$f_{[X,Y]}= \{f_X\,,\,f_Y\}_p$. Since $\tilde {[X,Y]}=\{\tilde {X},
\tilde {Y}\}_0$, we obtain \begin{equation*}\{\tilde {X}, \tilde
{Y}\}\,\circ \, \Psi=\{\tilde {X} \circ \Psi \,, \, \tilde {Y} \circ
\Psi \}_p.\end{equation*}  Hence $\Psi$ is a symplectomorphism.
\end{proof}

\par Now, we obtain an adapted Weyl correspondence on $O(\xi_0)$ by
transferring to $O(\xi_0)$ the Berezin-Weyl calculus on ${\bar
N}\times {\bar{\mathfrak n}}\times o(\xi_2)$. We say that a smooth
function $f$ on $O(\xi _0)$ is a symbol (resp. a  P-symbol, an
S-symbol) on ${O}(\xi _0)$ if $f\circ \Psi $ is a symbol (resp. a
P-symbol, an  S-symbol) for the Berezin-Weyl calculus on ${\bar
N}\times {\bar{\mathfrak n}}\times o(\xi_2)$ .

\begin{proposition} Let $\mathcal A$ be the space of all
P-symbols on $ O(\xi _0)$ and let $\mathcal B$ be the space of
differential operators on $C^{\infty}({\bar N},E)$. Then the map
${\mathcal W}: {\mathcal A}\rightarrow {\mathcal B}$ defined by the
${\mathcal W}(f)=W(f\circ \Psi )$ is an adapted Weyl correspondence
in the sense of \cite{CaDGA}, Section 6.1, that is, the map
${\mathcal W}$ satisfies the following properties

\begin{enumerate}

\item The map $\mathcal W$ is a linear isomorphism from $\mathcal A$
onto $\mathcal B$;

 \item the elements of $\mathcal B$ preserve a fixed dense domain $D$ of
 $L^2({\bar N},E)$;

\item the constant function $1$ belongs to $\mathcal A$, the identity
operator $I$ belongs to $\mathcal B$ and ${\mathcal W}(1)=I$;

\item $A\in {\mathcal B}$ and $B\in {\mathcal B}$ implies $AB\in
{\mathcal B}$;

\item for each $f$ in $\mathcal A$ the complex conjugate $\bar {f}$
of $f$ belongs to ${\mathcal A}$ and the adjoint operator ${\mathcal
W}(f)^{\ast}$ is an extension of ${\mathcal W}({\bar f})$;

\item the elements of $D$ are $C^{\infty}$-vectors for the
representation ${\pi}$, the functions $\tilde {X}$ ($X\in {\mathfrak
g}$) are in $\mathcal A$ and ${\mathcal W}(i\tilde X)\,\phi=d{\pi}
(X)\phi$ for each $X\in {\mathfrak g}$ and each $\phi \in D$.
\end{enumerate}
\end{proposition}
\begin{proof} The properties (1)-(4) are satisfied with $D=C_0({\bar N},E)$.
Property (5) is a consequence of (2) of Proposition 4.1. Property
(6) follows from Proposition 5.1.
\end{proof}

\section {Adapted Weyl correspondence for $\pi_0$}
In this section, we use the same method as in Section 5 to get a
symplectomorphism from ${\bar N}\times {\bar{\mathfrak n}}\times
o(\xi_2)$ onto a dense open subset of the orbit
$O(\xi_1,\xi_2)\subset {\mathfrak g}_0$ and then to construct an
adapted Weyl correspondence on $O(\xi_1,\xi_2)$.

\begin{proposition} Let $\Psi_0$ be the map from ${\bar N}\times
{\bar{\mathfrak n}}\times o(\xi_2)$ to ${\mathfrak g}_0$ defined by
\begin{equation*}\Psi_0 (y,Z,\varphi)=\Bigl(\Ad ({\tilde k}(y))
\,\xi_1\,,\,p_{\mathfrak k}^c
\bigl( \Ad (y)(\varphi-\theta (Z))\bigr)\Bigr).\end{equation*} Then,
for each $(v,U)\in {\mathfrak g}_0$, the Berezin-Weyl symbol of the
operator $-id\pi (X)$ is the P-symbol $f_{(v,U)}$ defined by
\begin{equation*}f_{(v,U)}(y,Z,\phi)=\langle \Psi_0
(y,Z,\varphi)\,,\,(v,U)\rangle.\end{equation*} \end{proposition}

\begin{proof} The proof is quite similar to that of Proposition 5.1. Let $(v,U)$ be
an element of ${\mathfrak g}_0$. By using the explicit expression
for $-id\pi_0(u,V)$ given in Proposition 3.3 and Lemma 4.2, we
obtain
\begin{align*}&f_{(v,U)}(y,Z,\varphi)=(p_{\bar {\mathfrak n}}(\Ad
(y^{-1})U),Z)+\beta (\Ad ({\tilde k}(y))\,\xi_1,v)\\&+\beta (\varphi
, p_{\mathfrak m}(\Ad (y^{-1})U))=\beta (\Ad (y)(\varphi-\theta
(Z)),U)+\beta (\Ad ({\tilde k}(y))\,\xi_1,v). \end{align*} This
gives the result.
\end{proof}

\par We denote by $\omega_0$ the Kirillov 2-form on $O(\xi_1,\xi_2)$
and by $\{\cdot,\cdot\}_0$ the corresponding Poisson bracket. Let
${\tilde O}(\xi_1,\xi_2)$ denote the dense open subset of
$O(\xi_1,\xi_2)$ defined by
\begin{equation*}{\tilde O}(\xi_1,\xi_2)=\{(v,k)\cdot
(\xi_1,\xi_2):v\in V,\,k\in K\cap {\bar N}MAN\}.\end{equation*}

\begin{proposition} The map $\Psi_0$ is a symplectomorphism from
$({\bar N}\times {\bar{\mathfrak n}}\times
o(\xi_2),\omega_p=d\alpha\otimes \omega_2)$ onto $({\tilde
O}(\xi_1,\xi_2),\omega_0)$.
\end{proposition}

\begin{proof} First, we show that for each $\xi \in {\tilde
O}(\xi_1,\xi_2)$ there exists a unique $(y,Z,\varphi)\in {\bar
N}\times {\bar{\mathfrak n}}\times o(\xi_2)$ such that $\Psi_0
(y,Z,\varphi)=\xi$. Let $\xi \in {\tilde O}(\xi_1,\xi_2)$. Then we
can write $\xi=(v,k)\cdot (\xi_1,\xi_2)$ with $v\in V$ and $k\in
K\cap {\bar N}MAN$. Clearly, the equation $\Psi_0 (y,Z,\varphi)=\xi$
is equivalent to

\begin{equation*} \left \lbrace \begin{aligned} (a)&\quad \Ad ({\tilde k}(y))\,\xi_1
=\Ad (k) \,\xi_1  \\ (b)&\quad p_{\mathfrak k}^c \bigl( \Ad
(k^{-1}y)(\varphi-\theta (Z))\bigr)=\xi_2+[\Ad (k^{-1})v,\xi_1].
\end{aligned}\right.\end{equation*}
Equation (a) determines $y$ uniquely. Moreover, $m:=k^{-1}{\tilde
k}(y)$ is an element of $M$. We set $n'(y)=m{\tilde a}(y){\tilde
n}(y){\tilde a}(y)^{-1}m^{-1}$. Then $n'(y) \in N$ and we have
\begin{equation*} k^{-1}y=k^{-1}{\tilde k}(y){\tilde a}(y){\tilde n}(y)
=m{\tilde a}(y){\tilde n}(y)=n'(y)m{\tilde a}(y).\end{equation*}
Thus, setting $Y:=\Ad (n'(y)m)\varphi-\Ad (m)\varphi \in {\mathfrak
n}$, we can write
\begin{equation*}p_{\mathfrak k}^c \bigl( \Ad
(k^{-1}y)(\varphi-\theta (Z))\bigr)=\Ad (m)\varphi+ p_{\mathfrak
k}^c \bigl(Y-\Ad (n'(y){\tilde a}(y)m)\theta
(Z)\bigr).\end{equation*} Hence, using Lemma 3.1, we see that
Equation (b) is equivalent to
\begin{equation*} \left \lbrace \begin{aligned} (c)&\quad \Ad
(m)\,\varphi=\xi_2   \\ (d)&\quad p_{\mathfrak k}^c \bigl(Y-\Ad
(n'(y){\tilde a}(y)m)\theta (Z)\bigr)=[\Ad (k^{-1})v,\xi_1].
\end{aligned}\right.\end{equation*}
Finally, we get $\varphi=\Ad (m^{-1})\xi_1$ and, by using Lemma 3.1
again, we see that there exists a unique element $Z$ of $\bar
{\mathfrak n}$ satisfying Equation (d). This proves the existence of
a unique element $(y,Z,\varphi)$ satisfying $\Psi
(y,Z,\varphi)=\xi$.
\par In the same way, we show that $\Psi_0$ takes values in
${\tilde O}(\xi_1,\xi_2)$ and we can conclude that $\Psi_0$ is a
bijection from ${\bar N}\times {\bar{\mathfrak n}}\times o(\xi_2)$
onto ${\tilde O}(\xi_1,\xi_2)$.
\par By following the same method as in the proof of Proposition
5.2, we show that $\Psi_0^{\ast}(\omega_0)=\omega_p$. Since the
2-form $\omega_p$ is non-degenerate, this also shows that $\Psi_0$
is regular. Finally, $\Psi_0$ is a symplectomorphism. \end{proof}

\par We can define the notion of symbols (P-symbols,
S-symbols) on $O(\xi_1,\xi_2)$ as in Section 5. Then we obtain the
following proposition which is analogous to Proposition 5.3.

\begin{proposition} Let ${\mathcal A}_0$ be the space of all
P-symbols on $ O(\xi _1,\xi_2)$ and let $\mathcal B$ be the space of
differential operators on $C^{\infty}({\bar N},E)$. Then the map
${\mathcal W}_0: {\mathcal A}_0\rightarrow {\mathcal B}$ defined by
the ${\mathcal W}_0(f)=W(f\circ \Psi_0 )$ is an adapted Weyl
correspondence in the sense of \cite{CaDGA}, Section 6.1.
\end{proposition}

\section {The Dooley-Rice contraction
revisited} In this section, we introduce the Dooley-Rice contraction
maps from $G_0$ to $G$ and we show how to use the results of the
previous sections in order to get a new version of Theorem 1 of
\cite{DoR2} for the \lq noncompact\rq\,  realizations of the
representations.

\par We consider the family of maps $c_r:\,G_0\rightarrow G$
defined by \begin{equation*}
c_r(v,k)=\exp_G(rv)\,k\end{equation*}for $v\in V$, $k\in K$ and
indexed by $r\in ]0,1]$. One can easily show that
\begin{equation*}\lim_{r\rightarrow
0}\,c_r^{-1}(c_r(g)\,c_r(g'))=g\,g'\end{equation*} for each
$g,\,g'\in G_0$. Then the family $(c_r)$ is a group contraction of
$G$ to $G_0$ in the sense of \cite{MN} (see also \cite{CaJAM}).

\par Let $(\xi_1,\xi_2)\in {\mathfrak g}_0$ as in Section 3. Recall
that $\pi_0$ is a unitary irreducible representation of $G_0$
associated with $(\xi_1,\xi_2)$. For each $r\in ]0,1]$, we set
$\xi_r:=(1/r)\xi_1+\xi_2$ and we denote by $\pi_r$ the principal
series representation of $G$ corresponding to $\xi_r$. Then we have
the following contraction result which is analogous to \cite{DoR2},
Theorem 1.
\begin{proposition} For each $(v,k)\in G_0$, $\phi \in C_0({\bar
N},E)$ and $y\in {\bar N}$, we have \begin{equation*}
\lim_{r\rightarrow 0}\,\pi_r(c_r(v,k))\phi\,(y)=\pi_0(v,k)\,\phi
(y).\end{equation*} \end{proposition} \begin{proof} By taking into
account the explicit expressions for $\pi_r$ and $\pi_0$ given in
Section 2 and Section 3 (Formulas (2.3) and (3.2)), we have just to
verify that
\begin{equation*}\lim_{r\rightarrow 0}\,{\frac {1}{r}}\beta
\Bigl(\xi_1\,,\,\log {\tilde a}(y)-\log {\tilde a}(k^{-1}\exp
(-rv)y)\Bigr)=\beta \bigl( \Ad ({\tilde k}(y))\xi_1,v
\bigr).\end{equation*} But applying Lemma 2.1 we have
\begin{align*} {\frac {d}{dt}}\,&{\tilde a}(y)^{-1}{\tilde a}
(k^{-1}\exp (-rv)y)\vert_{t=0} \\&= {\frac {d}{dt}}\,{\tilde
a}^{-1}(k^{-1}y){\tilde a} (\exp (-r\Ad
(k^{-1})v)k^{-1}y)\vert_{t=0}\\&=-{\tilde p}_{\mathfrak a}\bigl(
\Ad( {\tilde k}(k^{-1}y)^{-1})\Ad (k^{-1})v\bigr)\\&=-{\tilde
p}_{\mathfrak a}\bigl(\Ad ({\tilde k}(y)^{-1})\,v \bigr).\\
\end{align*} Then we obtain
\begin{align*}\lim_{r\rightarrow 0}\,&{\frac {1}{r}}\beta
\Bigl(\xi_1\,,\,\log {\tilde a}(y)-\log {\tilde a}(k^{-1}\exp
(-rv)y)\Bigr)=\beta \bigl( \xi_1\,,\,{\tilde p}_{\mathfrak
a}\bigl(\Ad ({\tilde k}(y)^{-1})\,v \bigr)\\&=\beta \bigl(\xi_1\,,\,
\Ad ({\tilde k}(y)^{-1})\,v \bigr)=\beta \bigl(\Ad ({\tilde
k}(y))\,\xi_1\,,\,v\bigr).\\ \end{align*} The result follows.
\end{proof}

\section {Contraction of adapted Weyl correspondences}
For each $r\in ]0,1]$, we denote by $\Psi_r$ the symplectomorphism
from ${\bar N}\times {\bar {\mathfrak n}}\times o(\xi_2)$ onto
${\tilde O}(\xi_r)$ introduced in Section 5 and by ${\mathcal W}_r$
the adapted Weyl correspondence on $O(\xi_r)$. In this section, we
show how the symplectomorphisms $\Psi_r$ contract to the
symplectomorphism $\Psi_0: \,{\bar N}\times {\bar {\mathfrak
n}}\times o(\xi_2)\rightarrow {\tilde O}(\xi_1,\xi_2)$ and how the
correspondences ${\mathcal W}_r$ contract to ${\mathcal W}_0$.

\par For each $r\in ]0,1]$, we denote by $C_r$ the differential of
$c_r$. Then the family $(C_r)$ is a contraction of Lie algebras from
${\mathfrak g}$ onto ${\mathfrak g}_0$, that is,
\begin{equation*}\lim_{r\rightarrow 0}\,C_r^{-1}\bigl(
[C_r(X)\,,\,C_r(Y)]\bigr)=[X\,,\,Y]_0\end{equation*} for each $X,\,Y
\in {\mathfrak g}_0$. We also denote by $C_r^{\ast}:{\mathfrak
g}^{\ast}\simeq {\mathfrak g}\rightarrow {\mathfrak
g}_0^{\ast}\simeq {\mathfrak g}_0$ the dual map of $C_r$.

\begin{proposition} For each $(y,Z,\varphi)\in {\bar N}\times {\bar {\mathfrak
n}}\times o(\xi_2)$, we have
\begin{equation*}\lim_{r\rightarrow 0}\,C_r^{\ast}\bigl(\Psi_r
(y,Z,\varphi)\bigr)=\Psi_0(y,Z,\varphi).\end{equation*}\end{proposition}
\begin{proof} Let $(v,U)\in {\mathfrak g}_0$. Since $\mathfrak k$ and $V$ are orthogonal
with respect to $\beta$, we have
\begin{align*} \langle C_r^{\ast}&\bigl(\Psi_r
(y,Z,\varphi)\bigr)\,,\,(v,U)\rangle=\langle \Psi_r
(y,Z,\varphi)\,,\,C_r(v,U)\rangle \\&=\langle (1/r)\Ad ({\tilde
k}(y))\xi_1+\Ad(y)(\varphi-\theta (Z))\,,\,rv+U \rangle
\\&=\beta \bigl(\Ad
({\tilde k}(y))\xi_1 ,v\bigl)+r\beta \bigl(p_V^c(\Ad
(y)(\varphi-\theta(Z)),v\bigr)\\&+\beta \bigl(p_{\mathfrak k}^c(\Ad
(y)(\varphi-\theta(Z)),U\bigr)\\ \end{align*} Then \begin{align*}
\lim_{r\rightarrow 0}\,\langle C_r^{\ast}&\bigl(\Psi_r
(y,Z,\varphi)\bigr)\,,\,(v,U)\rangle\\&=\langle \Ad ({\tilde
k}(y))\xi_1 +p_{\mathfrak k}^c\bigl(\Ad
(y)(\varphi-\theta(Z)\bigr)\,,\,(v,U)\rangle.\end{align*} Hence the
result.
\end{proof}
\par Now, let $f:\,O(\xi_1,\xi_2)\rightarrow {\mathbb C}$ be a
P-symbol of degree $d$, that is,
\begin{equation*}\widehat {f\circ \Psi_0}(y,Z)=\sum_{\vert
\alpha \vert \leq d}u_{\alpha}(y)Z^{\alpha} \end{equation*} where
each $u_{\alpha}$ is in $C^{\infty}({\bar N})$. Following
\cite{CDo}, we say that a family $f_r:\,O(\xi_r)\rightarrow {\mathbb
C}$ of symbols approximates $f$ if each $f_r$ is a P-symbol of
degree less than or equal to $d$, that is,
\begin{equation*}\widehat{f_r\circ \Psi_r}(y,Z)=\sum_{\vert
\alpha \vert \leq d}u_{\alpha}^r(y)Z^{\alpha} \end{equation*} and
if, for each $\alpha$, $u_{\alpha}^r-u_{\alpha}$ and all its
derivatives $\partial_{\gamma}(u_{\alpha}^r-u_{\alpha})$ converge
uniformly on compacts to zero, as $r\rightarrow 0$. Here, for each
$v\in C^{\infty}({\bar N})$ and each $\gamma
=(\gamma_1,\gamma_2,\ldots,\gamma_n)$, the derivative
$\partial_{\gamma}v$ is defined by
\begin{align*}\partial_{\gamma}v(y)=\left( \frac
{d}{dt_1}\right)^{\gamma_1}\left( \frac {d}{dt_2}\right)^{\gamma_2}
\ldots \left( \frac {d}{dt_n}\right)^{\gamma_n}&\bigl( v(y\exp
(t_1E_1)\exp (t_2E_2)\ldots\\&\ldots \exp (t_nE_n)\bigr)\bigr
\vert_{t_1=t_2=\ldots=t_n=0}.\end{align*} By using the properties of
the Berezin-Weyl calculus, we immediately obtain the following
proposition.
\begin{proposition} Let $f$ be a P-symbol on $O(\xi_1,\xi_2)$.
Let $(f_r)$ be a family of P-symbols which approximates $f$. Then,
for each $\phi \in C_0({\bar N},E)$ and each $y\in {\bar N}$, we
have
\begin{equation*}\lim_{r\rightarrow 0}\,({\mathcal W}_r(f_r)\phi) (y)=
({\mathcal W}_0(f)\phi) (y). \end{equation*} \end{proposition}

Then we can deduce a contraction result for the derived
representations from the contraction of the symplectomorphisms
$\Psi_r$ to $\Psi_0$.
\begin{proposition} 1) Let $(v,U)\in {\mathfrak g}_0$. Then the family
$(\widetilde {C_r(v,U)})_{r\in ]0,1]}$ approximates $\widetilde
{(v,U)}$.

\noindent 2) For each $(v,U)\in {\mathfrak g}_0$, $\phi \in
C_0({\bar N},E)$ and $y\in {\bar N}$, we have
\begin{equation*} \lim_{r\rightarrow 0}\,\bigl( d\pi_r(C_r(v,U))\phi\bigr) \,(y)
=\bigl(d\pi_0(v,U)\phi\bigr)\,(y).\end{equation*}\end{proposition}

\begin{proof} 1) This follows from Proposition 8.1.

\noindent 2) Taking Proposition 5.3 and Proposition 6.3 into
account, the result is an immediate consequence of Proposition 8.2.
\end{proof}

\end{document}